\documentclass[12pt]{amsart}
\usepackage{epsfig, amsmath, amssymb, latexsym,  amscd}
\usepackage{amsthm}
\usepackage{xcolor}

\newcommand{ \pp }{{\mathbb P}}

\newcommand{ \zz}{{\mathbb Z}}

\newcommand{ \CC}{{\mathcal C}}
\newcommand{ \OO}{{\mathcal O}}

\def\cocoa{{\hbox{\rm C\kern-.13em o\kern-.07em C\kern-.13em o\kern-.15em A}}}

\newcommand{\rank}{\operatorname{rank}}

\newtheorem{theorem}{Theorem}[section]
 \newtheorem{corollary}[theorem]{Corollary}
 \newtheorem{lemma}[theorem]{Lemma}
 \newtheorem{proposition}[theorem]{Proposition}

 \newtheorem{conjecture}[theorem]{Conjecture}
  \theoremstyle{definition}
 \newtheorem{question}[theorem]{Question}
 \newtheorem{definition}[theorem]{Definition}
 \theoremstyle{remark}
 \newtheorem{remark}[theorem]{Remark}
  \newtheorem{example}[theorem]{Example}

\begin{document}

\title[Sets of points which project to complete intersections]{Sets of points which project to complete intersections, and unexpected cones}
\date{}
\author[]{Luca Chiantini}
\address{Luca Chiantini. Dipartimento di Ingegneria dell'Informazione e Scienze Matematiche, Universit\`a di Siena, Italy}
\email{luca.chiantini@unisi.it}
\author[]{Juan Migliore}
\address{Juan Migliore. Department of Mathematics, University of Notre Dame, Notre Dame, IN 46556 USA}
\email{migliore.1@nd.edu}

\begin{abstract}
The paper is devoted to the description of those non-degenerate sets of points $Z$ in $\mathbb P^3$ whose general projection to a general plane is a complete intersection of curves in that plane. One large class of such $Z$ is what we call $(a,b)$-grids. We relate this problem to the {\em unexpected cone property} $\CC(d)$, a special case of the unexpected hypersurfaces which have been the focus of much recent research. After an analysis of  $\CC(d)$ for small $d$, we show that a non-degenerate set of $9$ points has a general projection that is the complete intersection of two cubics if and only if the points form a $(3,3)$-grid. However, in an appendix we describe a set of $24$ points that are not a grid but nevertheless have the projection property. These points arise from the $F_4$ root system. Furthermore, from this example we find subsets of $20$, $16$ and $12$ points with the same feature. 
\end{abstract}

\thanks{
{\bf Acknowledgements}: 
Migliore was partially supported by Simons Foundation grant \#309556. Chiantini was partially supported by the Italian INdAM-GNSAGA. The authors of the appendix thank the Centro Internazionale per la Ricerca Matematica (CIRM), which supported the Workshop on Lefschetz Properties and Jordan Type in Algebra, Geometry and Combinatorics where that work was initiated. We also thank the referee for the careful reading of our paper, and for the useful comments.
}



\maketitle
\thispagestyle{empty}


\section{Introduction}
The main purpose of  this paper is to introduce a systematic study of a problem on the projective geometry of finite sets, namely the classification of 
sets of points, $Z$, in $\mathbb P^3$ whose general projections $\pi(Z)$ to a  plane are  complete intersections. What makes this problem especially interesting is that (as we shall see) outside of a certain class of such $Z$, this property seems to be extremely rare. And yet, in the appendix we give a highly non-trivial example of just such a set. In studying this simple-sounding property, one quickly sees deep connections to  interpolation theory and to the classification of {\it unexpected hypersurfaces}.

If $Z \subset \mathbb P^3$ is already a planar complete intersection then so is its general projection, so we assume that $Z$ is non-degenerate. Looking for more examples, one immediately thinks of four general points, and after a little thought one extends this to any set of points arising as the intersection of a set of $a$  lines in one ruling of a smooth quadric surface with a set of $b$ lines on the other ruling (which below we will call an $(a,b)$-grid). (The case $a=2$ actually does not need to lie on a quadric, but it still follows the same pattern). Trying to extend this in turn, one could try to study sets of $ab$ points arising as the intersection of a curve of degree $a$ and a curve of degree $b$ meeting in those $ab$ points in  $\mathbb P^3$. However, building on a result of Diaz (\cite{D86}) and of Giuffrida (\cite{G88}), at the beginning of section \ref{grid section} we  show that the $(a,b)$-grids are the only such examples. (See Proposition \ref{extend DG})  And yet a spectacular example arose almost by accident in a recent workshop, and this example is detailed in an appendix below, written by all participants of our work group. 
This paper arose as a first step to understanding how such a phenomenon is possible, and the connection to unexpected hypersurfaces (in particular unexpected cones) was a natural outgrowth.

Unexpected hypersurfaces have been studied fairly extensively recently (and in a sense, also classically), in different settings. On a classical level, special linear systems have been studied for more than a century in countless settings. Questions of maximal rank curves (e.g.  \cite{BE85},  \cite{HH82}, and more recently \cite{BDSSS19}) are closely related. The famous SHGH Conjecture \cite{Seg61,Har86,Gim87,Hir89} was one of the motivations for the more modern work on unexpected hypersurfaces. More recently,  for instance, \cite{CHMN18}, \cite{HMNT}, \cite{DMO}, \cite{Szpond}, \cite{HMT} have all explored this new direction.

For us, we will say that a reduced subscheme $V \subset \mathbb P^n$ {\em admits an unexpected hypersurface of degree $d$ and multiplicity $m$} if, for a general point $P \in \mathbb P^n$ we have
\[
\dim [I_V \cap I_P^m]_d > \max \left \{ 0, \dim [I_V]_d - \binom{m-1+n}{n} \right \}.
\]
In the literature it sometimes happens that $P$ is replaced by a union of general points (and then the multiplicities can differ among the points), or even a general variety (usually linear) of higher dimension. The first extended study of unexpected hypersurfaces came in \cite{CHMN18} where $V$ was a finite set of points in $\mathbb P^2$ and $m = d-1$. In this paper $V$ will almost always be a finite set of points, and we will always have $m = d$ ({\em unexpected cones}). We refer to the existence of unexpected surfaces with $m=d$ as {\em Property} $\CC(d)$, {\em the unexpected cone property}.

Suppose $Z \subset \mathbb P^3$ is a finite set of points whose general projection is the complete intersection in the plane of curves $C_1,C_2$ of degrees $a,b$ respectively. Let $P \in \mathbb P^3$ be a general point. Clearly $Z$ lies on the cone over $C_1$ and the cone over $C_2$, both with vertex $P$. Since having the general projection be a complete intersection is such a rare property, it is not surprising that it turns out that these cones are unexpected in the above sense. Thus in this paper we necessarily have to focus largely on the properties $\CC(d)$. 
The unexpected cone property usually holds for finite sets in $\pp^3$  only when their
general projection is in special position. In particular, it turns out that there is a strict
connection between grids, unexpected cone properties, and finite sets in $\pp^3$
whose general projection is a complete intersection.  These connections are somewhat subtle (see Remark \ref{ptci}), and are explored further in section \ref{geom section}.

Fix integers $a,b$ with $ a\leq b$. As indicated above, a set  $Z \subset \mathbb P^3$
of $ab$ points is called an {\em $(a,b)$-grid} if there is a set of $a$ pairwise skew lines $L_1,\dots,L_a$ and a set of $b$ pairwise skew lines $L_1',\dots,L_b'$ 
such that each $L_i$ meets each $L_j'$, and $Z$ is the set of the intersection points $L_i\cap L_j'$, $i=1,\dots a,\ \  j=1,\dots,b$.
As remarked above, the general projection of an $(a,b)$-grid is a complete intersection of type $(a,b)$. Indeed, this was noticed for instance by D. Panov, 
in his reply to a question raised  by F. Polizzi -- see {\tt MathOverflow} Question 67265 in \cite{CDFPGR}. 

Grids are also examples of finite sets which admit a special class of unexpected surfaces. Indeed, we show in Proposition \ref{CC2}  that a non-degenerate set $Z$ of at least $6$ points satisfies $\CC(2)$ if and only if $Z$ is contained in a pair of skew lines, each containing at least $3$ points of $Z$. In section \ref{C(3) section} we also carefully analyze $\CC(3)$ for grids.

The main results of the paper are 
Theorem \ref{gridex}, which study the properties $\CC(d)$ for grids, Proposition \ref{CC2} which characterizes property $\CC(2)$ for finite sets of points, and Theorem \ref{proj CI33}, which says that if $Z$ is a non-degenerate set of $9$ points whose general projection to a plane is a complete intersection 
of cubics then $Z$ must be a $(3,3)$-grid.

As noted above, the existence of sets whose general projection is a complete intersection is quite rare, and
our analysis suggests that such sets are in special position with respect to linear subspaces. Recall that a set of points $Z \subset \mathbb P^3$ 
such that no $3$ points of $Z$ are collinear and no $4$ points of $Z$ are coplanar are said to be in 
{\it  general position} in the language of \cite{harris} page 7, although we will use the term {\em linear general position}, 
to avoid confusion with a general set of points. The work in this paper leads us to make the following conjecture.

\begin{conjecture}
If $Z \subset \mathbb P^3$ is a set of at least $5$ points in linear general position then the general projection of $Z$ to a plane is not a complete intersection.
\end{conjecture}

We are not aware of any examples of non-degenerate sets of points in $\mathbb P^n$, $n \geq 4$, whose general projection to $\mathbb P^{n-1}$ is a complete intersection.
\smallskip

The outline of the paper is as follows. In the second section we collect some background, definitions and facts that we will use in this paper. In section \ref{grid section} we give some facts about grids related to the unexpected cone properties, culminating in Theorem \ref{gridex}. In section \ref{C(2) section} we classify sets of points with property $\CC(2)$. Section \ref{C(3) section} considers sets of points with property $\CC(3)$, as well as the projection property. The main result is Theorem \ref{proj CI33}, which says that if $Z$ is a non-degenerate set of $9$ points whose general projection to a plane is a complete intersection of cubics then $Z$ must be a $(3,3)$-grid. Section \ref{geom section} explores further the connections between the projection property and the unexpected cone properties. 

The main purpose of section \ref{non-grid ex section} (Appendix) is to give a very surprising (dare we say unexpected?) example of $24$ non-degenerate points that do not form a $(4,6)$-grid, but whose general projection is nevertheless a complete intersection of type $(4,6)$. Furthermore, we find subsets of $20$, $16$ and $12$ points, again not  grids, whose general projection is a complete intersection. The original example of $24$ points arose in a work group at the Workshop on Lefschetz Properties and Jordan Type in Algebra, Geometry and Combinatorics in Levico Terme in 2018 after a careful study of an observation in \cite{HMNT}, and the participants of that work group are all authors of the appendix.


\section{Notation and background} \label{background}
Throughout this paper we work over  projective spaces with base field  $\mathbb K$ of characteristic~$0$.

Let $Z \subset \mathbb P^n$ be a reduced, zero-dimensional scheme. We set  $\ell(Z)=$ cardinality of $Z$. If $P \in \mathbb P^n$, we denote by $mP$ the scheme defined by the ideal $I_P^m$. 

\begin{definition}
Let $V \subset \mathbb P^n$ be a reduced, non-degenerate subscheme of dimension $\leq n-2$. 
We say that $V \subset \mathbb P^n$ {\em admits an unexpected hypersurface of degree $d$ and multiplicity $m$} if, for a general point $P \in \mathbb P^n$, we have 
$$
\dim [I_V \cap I_P^m]_d > \max \left \{ 0, \dim [I_V]_d - \binom{m-1+n}{n} \right \}.
$$
That is, $V$ admits an unexpected hypersurface with respect to $P$ of degree $d$ and multiplicity $m$ if the $\binom{m-1+n}{n}$ conditions imposed by $mP$
on forms of degree $d$ vanishing on $V$ are not independent. 
When the righthand side of the above inequality is equal to 0, any element of the linear system corresponding to $[I_V \cap I_P^m]_d$ is called an {\em unexpected hypersurface}. When $\dim [I_V \cap I_P^m]_d  = 1$, we say that the unexpected hypersurface is {\em unique}.
\end{definition}

\begin{remark}
It is worth noting that $\dim [I_V]_d$ is taken as given, and there is no concern for whether this dimension is unexpected in any sense or not. For instance, $V$ may be a finite set of points that do not impose independent conditions on hypersurfaces of degree~$d$, and a priori this does not influence whether $V$ admits an unexpected hypersurface or not.
\end{remark}

\noindent In this paper our primary focus will  be for the case where $V$ is a reduced, zero-dimensional scheme in $\mathbb P^3$, and for emphasis we will denote this scheme by $Z$ rather than $V$. In this case we refer to {\em unexpected surfaces}.

\begin{proposition}[\cite{HMNT}] \label{cone in P3}
Let $C$ be a reduced, equidimensional, non-degenerate curve of degree $d$ in $\mathbb P^3$ ($C$ may be reducible, disconnected, and/or singular 
but note that $d \geq 2$ since $C$ is non-degenerate, with $C$ being two skew lines if $d=2$). Let $P \in \mathbb P^3$ be a general point. 
Then the cone over $C$ with vertex $P$ is an unexpected surface of degree $d$ for $C$ with multiplicity $d$ at $P$. It is the unique unexpected surface of this degree and multiplicity.
\end{proposition}

\begin{remark} \label{pts-curves}
As noted in \cite{HMNT} Corollary 2.5, if $Z \subset C$ is a set of points such that $[I_C ]_d = [I_Z ]_d$ then the cone over $C$ is an unexpected surface of degree $d$ and multiplicity $d$ also for $Z$.
\end{remark}

The extremal case, in which the unexpected surface has degree $d$ and multiplicity $d$ at $P$, is the main target of our analysis. We introduce, 
in this case, the notion of the \emph{unexpected cone property}.

\begin{definition}
Let $Z \subset \mathbb P^3$ be a reduced, non-degenerate zero-dimensional subscheme and let $d$ be a positive integer. We say that $Z$ has the {\em unexpected cone property} $\CC(d)$ if, for a general point $P \in \mathbb P^3$, there is an {\em unexpected} surface $S_P$ of degree $d$ and multiplicity $d$ at $P$ 
containing the points of $Z$. 
\end{definition}

\noindent Notice that by Bezout's theorem, $S_P$ must be a cone with vertex $P$ over some curve $C$ of degree $d$ (at the least, it is the cone over a hyperplane section of $S_P$, but there can be other such curves).  The delicate conditions are the very existence of $S_P$ containing $Z$ (for each general $P$) and that it be unexpected (which is a numerical condition). Notice also that since we require $Z$ to be non-degenerate, we must have $\ell(Z) \geq 4$.

Families of cones passing through $Z$, with general vertex, are related to families of plane curves passing through a general plane projection of $Z$.
Namely, $dim [I_Z \cap I_P^d]_d$ is equal to the dimension of forms of degree d containing the projection.

The unexpectedness of the cones implies that a plane projection of $Z$ from a general point $P$ has some peculiarities in its postulation. For instance, we have the following.

\begin{proposition} \label{indep cond}
If $Z \subset \mathbb P^3$ is a set of points satisfying $\CC(d)$ then the general projection $\pi_P(Z)$ of $Z$ to a plane is a set of $\ell(Z)$ points that do not impose independent conditions on the linear system of plane curves of degree $d$.
\end{proposition}

\begin{proof}
The condition $\CC(d)$ implies
\[
\dim [I_Z]_d - \binom{d+2}{3} < \dim [I_Z \cap I_P^d]_d = \dim [I_{\pi_P(Z)}]_d.
\]
Suppose that $\pi_P(Z)$ imposes independent conditions on the linear system of plane curves of degree $d$, so $\dim [I_{\pi_P(Z)}]_d = \binom{d+2}{2} - \ell(Z)$. This means
\[
\dim [I_Z]_d < \binom{d+2}{3} + \binom{d+2}{2} - \ell(Z) = \binom{d+3}{3} - \ell(Z).
\]
Thus $\ell(Z) < \binom{d+3}{3} - \dim [I_Z]_d$. But the latter is the Hilbert function of $Z$ in degree $d$, which is always less than or equal to $\ell(Z)$ since $Z$ is zero-dimensional. Contradiction.
\end{proof}

\begin{remark}
The contrapositive of Proposition \ref{indep cond} is worth stating: If the general projection $\pi_P(Z)$ is a set of points that impose independent conditions on curves of degree $d$ then $Z$ admits no unexpected cone of degree $d$.
\end{remark}

We will often use, in the sequel, the definition and the properties of the Hilbert functions of finite sets in projective spaces, and their
difference functions.
 
\begin{definition} For a finite set $Z$ of points in a projective space $\pp^n$ define the Hilbert function $h_Z: \zz\to\zz$ as the rank of the natural restriction map,
i.e. for all $i$:
$$  h_Z(i) = \rank H^0\OO_{\pp^n}(i) \to H^0\OO_Z(i).$$
Equivalently, 
\[
h_Z(i) = \dim_{ \mathbb K} [R/I_Z]_i.
\]
The difference of the Hilbert function is $\Delta h_Z(i)=h_Z(i)-h_Z(i-1)$. The \emph{$h$-vector } of $Z$ is the set of non-zero values of $\Delta h_Z(i)$
\end{definition}

The following properties are rather elementary and well known. We will use them  in our arguments below.

\begin{proposition}\label{elem} 
Let $Z \subset \mathbb P^n$ be a finite set of $d$ points. Then:

\begin{itemize}
\item[(a)] $h_Z(i)=0$ for all $i<0$; $h_Z(0)=1$; $h_Z(i) = |Z| = d$ for all $i \gg 0$. 
\item[(b)] $h_Z(i)\geq h_Z(i-1)$ (i.e.$\Delta h_Z(i)\geq 0$) for all $i$, with equality if and only if $h_Z(i) = d$  or $i<0$. 
\item[(c)] The following conditions are equivalent:

\begin{itemize}
\item[(i)] $h_Z(d-2) < d$;

\item[(ii)] $h_Z(d-2) = d-1$;

\item[(iii)] the points of $Z$ are collinear.

\end{itemize}

\item[(d)] Improving (a), we have $h_Z(i) = |Z| = d$ for all $i \geq d-1$.
\end{itemize}
\end{proposition}

\begin{proof}
These are all standard facts, but since we are not aware of a good reference we briefly comment on the proof. Part (a) follows because $R/I_Z$ is a standard graded Cohen-Macaulay algebra of depth = Krull dimension = 1, so its Hilbert polynomial is the constant $\deg (Z) = |Z|$. For a general linear form $L$ and any integer $i$ we have the exact sequence
\[
0 \rightarrow [R/I_Z]_{i-1} \stackrel{\times L}{\longrightarrow} [R/I_Z]_i \rightarrow [R/(I_Z,L)]_i \rightarrow 0.
\]
This implies the first part of (b). But since $R/(I_Z,L)$ is a standard graded algebra, once it is zero in degree $i$ it must be zero in all larger degrees, so by (a) we are done.  

For (c), (ii) trivially implies (i). Assume (i), so $Z$ imposes $h_Z(d-2) < d$ conditions on forms of degree $d-2$. Choose any point and relabel if necessary so that it is $P_d$. Any form of degree $d-2$ vanishing on $P_1,\dots,P_{d-1}$ must also vanish on $P_d$. Define $F$ as the product of a general linear form $L_{1,2}$ vanishing on $P_1$ and $P_2$ and a general linear form $L_i$ vanishing on each $P_i$ for $3 \leq i \leq d-1$. By generality none of the $L_i$ vanish on $P_d$ for $3 \leq i \leq d-1$, so $L_{1,2}$ must vanish on $P_d$. Since the point was arbitrarily chosen, it follows that all the points lie on the line spanned by $P_1$ and $P_2$, so we have (iii). Let $Z$ be a set of $d$ collinear points on a line $\lambda$. By Bezout's theorem, any form of degree $d-2$ vanishing on $Z$ must vanish on all the points of $\lambda$, so $h_Z(d-2) < d$. But (a) and (b) imply $h_Z(d-2) \geq d-1$, so we have (ii). Finally, (d) follows from (c) (ii) and (b).
\end{proof}

A fundamental property of the Hilbert function, that we will need below, is the following: 

\begin{theorem}\label{BGM} Let $Z$ be a finite set in $\pp^n$. Assume that, for some $i_0,k>0$ the Hilbert function
of $Z$ satisfies $\Delta h_Z(i_0)=\dots =\Delta h_Z(i_0+k)=1$. Then $\Delta h_Z(i)\leq 1$ for all $i\geq i_0$  and 
$Z$ contains a subset of $i_0+k$ collinear points.
\end{theorem}
\begin{proof} The fact is classically known. E.g. it is an easy consequence of [BGM94] Theorem 3.6 (e).
\end{proof}

We will often use the following property, which is an easy, well-known  consequence of the relations between
the resolutions of two linked sets of points (see for instance \cite{migliorebook} section 6.1).

\begin{proposition} \label{linkci} Let $Z_1,Z_2$ be disjoint, non-empty finite sets of points in $\pp^2$, such that $Z=Z_1\cup Z_2$ is a complete intersection of type
$(a,b)$ with defining ideal $(F,G)$.  Assume that $Z_1$ is a complete intersection of type $(a,c)$.  Then $Z_2$ is a complete intersection of type $(a,b-c)$ if and only if $F$ is part of a minimal generating set of $I_{Z_1}$.
\end{proposition}

\begin{lemma} \label{planar pts}
Let $Z \subset \mathbb P^3$ be a set of points that lies on a plane. Then $Z$ does not satisfy $\CC(d)$ for any $d$.
\end{lemma}

\begin{proof}
Let $\pi_P$ be the projection of $Z$ to a general plane, from a general point $P$. Then $\pi_P(Z)$ and $Z$ have the same Hilbert function $h_Z$. Since we are considering cones, we have $\dim [I_{\pi_P(Z)}]_d = \dim [I_Z \cap I_P^d]_d$. Then
\[
\begin{array}{l}
\displaystyle \dim [I_Z \cap I_P^d]_d = \dim [I_{\pi_P(Z)}]_d =  \binom{d+2}{2} - h_Z (d)  \\ \\
\displaystyle = \binom{d+3}{3} - \binom{d+2}{3} - h_Z(d) = \dim [I_Z]_d - \binom{(d-1)+3}{3}
\end{array}
\]
so $Z$ does not admit unexpected cones.
\end{proof}


\section{Generalities on grids} \label{grid section}

In this paper we are interested in studying those sets of points $Z$ in $\mathbb P^3$ whose general projection to a general plane is the complete intersection of  curves  in $\mathbb P^2$ of degrees $a$ and $b$ respectively. One way that this can clearly happen is if $Z$ is already the intersection in $\mathbb P^3$ of a curve $A$ of degree $a$ and a curve $B$ of degree $b$. We first show that there is essentially only one situation where this can happen, which we call ``grids," and then we will make a careful study of properties of grids.

Of course we exclude the trivial case where $A$ and $B$ both lie on the same plane. 
Diaz (\cite{D86}) and of Giuffrida (\cite{G88}) independently proved the following result.

\medskip

\begin{quotation}
{\em 
Let $A$ be a reduced curve of degree $a$ and let $B$ be a reduced curve of degree $b$ in $\mathbb P^3$. Assume that both curves are non-degenerate and irreducible. Then the number of intersection points, $m$, of $A$ and $B$ (counted without multiplicity) satisfies $m \leq (a-1)(b-1)+1$. If equality holds then there is a quadric surface $Q$ containing both $A$ and $B$. If furthermore $a \geq 4$ then $Q$ is smooth, and on $Q$ $A$ is a rational curve of type $(a-1,1)$ and $B$ is a rational curve of type $(1,b-1)$.
}
\end{quotation}

\medskip

\noindent

For our current purposes we would like to remove the assumption that $A$ and $B$ are irreducible, and ask when it can happen that $m = ab$. In keeping with the rest of the paper, we make a slight change in the notation for the curves and for the degrees.
We will show the following.

\medskip

\begin{proposition} \label{extend DG}
Let $A$ be a reduced curve of degree $a$ and let $B$ be a reduced curve of degree $b$ in $\mathbb P^3$, both not necessarily irreducible, and such that $A$ and $B$ share no components. Assume that $a \leq b$. Let $m$ be the number of intersection points of $A$ and $B$ (counted without multiplicity). Assume that $m = ab$. Then

\smallskip

\begin{enumerate}

\item[i)] If $a=1$ then $B$ is a disjoint union of plane curves, each of them coplanar with $A$.

\item[ii)] If $2 \leq a $ then either $A\cup B$ is a plane curve, or  both $A$ and $B$ are unions of pairwise skew lines. In the latter case, if we have $a\geq 3$ then there is a smooth quadric surface $Q$ that contains both $A$ and $B$, and $A$ is of type $(a,0)$ and $B$ is of type $(0,b)$ on~$Q$.

\end{enumerate}

\end{proposition}

\begin{proof}

The proof is based on the general observation that two reduced curves $C,D$, of degrees $c,d$ respectively, with no common components, cannot meet in more than $cd$ points.
To see this, just project generically $C\cup D$ to a plane.

Let $A = A_1 \cup \dots \cup A_s$ and $B = B_1 \cup \dots \cup B_t$ be the irreducible decompositions of $A$ and $B$, with $\deg A_i = a_i$, $\deg B_j = b_j$,
so that $a = a_1 + \dots + a_s$, $b = b_1 + \dots + b_t$. By the previous general remark,  $m = ab$ means that the intersections
$A_i\cap B_j$ are pairwise disjoint and each $A_i$ meets each $B_j$ in $a_i b_j$ points.

First assume that $A$ is a line, i.e. $A=A_1$ and $a=a_1=1$. For every $j$ pick a point $P_j\in B_j\setminus A$. The plane $\pi_j$ spanned by $A$ and $P_j$ meets $B_j$ in $b_j+1$ points,
thus it contains $B_j$. This proves part i).

Consider now the case $a>1$.

Assume there exists some component of $A$, say $A_1$, of degree $a_1>1$. Then we claim that $A\cup B$ is degenerate. First we show that $A_1$ is degenerate, by using that the fact that the intersection
of $A_1$ and $B_1$ contain $a_1b_1$ points.  Indeed, assume that $A_1$ is non-degenerate. $B_1$ cannot be a line, because this contradicts part i) of the proposition
(with $A,B$ replaced by $B_1,A_1$ respectively). If $B_1$
is degenerate, of degree $b_1>1$, then its plane cannot contain more than $a_1<a_1b_1$ points of $A_1$, and we get a contradiction again. If $B_1$ is non-degenerate,
then we get a contradiction with \cite{D86} or \cite{G88}. Thus $A_1$ is degenerate, and it spans a plane $\pi$. Since for any $j$ the intersection
$A_1\cap B_j$ contains $a_1b_j>b_j$ points, then $\pi$ contains all the components $B_j$, hence $B\subset \pi$. Since $b>1$, then for any $i$ the plane $\pi$
contains $a_ib>a_i$ points of the component $A_i$, hence $\pi$ contains $A$ too, and the claim holds.

Assume now that all the components of $A$ are lines. If two of these lines meet, then they span a plane $\pi$ and, arguing as above, we obtain that $\pi$ contains all the components of $A$ and $B$, i.e. it contains $A \cup B$.

It remains to consider the case in which $A$ is a union of $a>1$ pairwise disjoint lines. In this case, we can repeat the previous argument, with $A,B$ interchanged, and get that
also $B$ is a union of pairwise disjoint lines. If $a>2$, then there exists a smooth quadric $Q$ which contains three lines of $A$. Every line of $B$ meet $Q$ in at least three points, thus $B\subset Q$. Since $b\geq a>2$, all the lines of $A$ meet $Q$ in at least three points, thus also $A\subset Q$. Then ii) holds.
\end{proof}

\begin{definition} A set $Z$ of $ab$ points in $\pp^3$ is an {\em $(a,b)$-grid} if there exists a set of $a$ pairwise skew lines 
$L_1,\dots,L_a$ and a set of $b$ pairwise skew  lines $L_1',\dots,L_b'$, such that $Z=\{L_i\cap L_j' \mid i_1,\dots, a, \ j=1,\dots,b\}$
(in particular, for all $i,j$  the lines $L_i,  L_j'$ are different and incident).
\end{definition}

\begin{remark}
 It is trivial that an $(a,b)$-grid is also a $(b,a)$-grid. Thus throughout this paper we will adopt the convention (without loss of generality) that $a \leq b$.
\end{remark}

\begin{remark}
We first make some geometric observations.

\begin{enumerate}

\item The case $a=1$ is trivial, and we will always assume $a \geq 2$.

\medskip

\item If $a=b=2$ we have a set of four points, and conversely any set of four points is a $(2,2)$-grid. But clearly four general points in $\mathbb P^3$ do not admit an unexpected surface of any degree. Their general projection to $\mathbb P^2$ is, however, a complete intersection.

\medskip

\item If $(a,b) = (2,3)$, the points of $Z$ lie on the smooth quadric surface $Q$ determined by  $L_1',L_2',L_3'$  so the lines determining the $(2,3)$-grid all lie on $Q$.

\medskip

\item If $a = 2$ and $b \geq 4$ is arbitrary, it is not necessarily true that the lines determining the $(2,b)$-grid lie on a quadric surface, although of course $Z$ itself lies on a 4-dimensional (vector space) family of quadrics.

\medskip

\item If  $a \geq 3$, then the lines $L_j'$ meet the unique (smooth) quadric $Q$ which contains $L_1,L_2,L_3$ in at least
$3$ points, hence the lines $L_j'$ are contained in $Q$. Thus each line $L_j'$ is contained in $Q$. It follows that $Z$ is contained in a
quadric, and that the lines $L_i$ are in one ruling while the lines $L_j'$ are in the other. 
\end{enumerate}
\end{remark}

It turns out that for $(a,b)$-grids with $b>2$, the $\CC(a)$ property is automatic. 

\begin{theorem} \label{gridex}
Assume $b \geq a \geq 2$ and $b\geq 3$. Then any $(a,b)$-grid $Z$ satisfies  the unexpected cone property $\CC(a)$.

Furthermore, if $a,b \geq 3$ then $Z$ also satisfies the unexpected cone property $\CC(b)$.
\end{theorem}

\begin{proof} 
We first consider $\CC(a)$.
Let $C$  be the union of the lines $L_1,\dots,L_a$ and let $P$ be a general point. If $2 \leq a < b$, we have $[I_Z]_a = [I_C]_a$ by Bezout's theorem, so Remark \ref{pts-curves} and Proposition \ref{cone in P3} give the result that the cone over $C$ with vertex $P$ is an unexpected cone of degree $a$, so $\CC(a)$ holds. 

Now assume that $a =b \geq 3$. Denoting now by $D$ the union of the $2a$ lines, we note that $D$ is a complete intersection of type $(2,a)$. For a form $F$ of degree $d \leq a-1$, $F$ vanishes on $Z$ if and only if it vanishes on $D$. Thus the Hilbert function of $Z$ agrees with that of $D$ up to degree $a-1$, so the $h$-vector of $Z$ begins
\[
(1,3,5,7,\dots, 2a-1)
\]
(where the $2a-1$ occurs in degree $a-1$). Since $\ell(Z) = a^2$ and the sum of the above integers is $a^2$, this is the entire $h$-vector. It follows that $Z$ imposes independent conditions on forms of degree $a-1$ and above. Hence
\[
\dim [I_Z]_a = \binom{a+3}{3} - a^2.
\]
In general, a point of multiplicity $m$ is expected to impose $\binom{m+2}{3}$ conditions on forms of degree $m$. But then an easy calculation gives
\[
\dim [I_Z]_a - \binom{a+2}{3} = \frac{3a + 2 - a^2}{2} .
\]
If $a=b=3$, this gives that the expected dimension of $\dim [I_Z \cap I_P^3]_3$ is $1$, while both the cone over the $L_i$ and the cone over the $L_j'$ have multiplicity $a$ at $P$ and are hence unexpected. If $a=b \geq 4$, the above calculation gives that we do not expect any surface of degree $a$ and multiplicity $a$ at $P$, so again the cones are unexpected. Hence $Z$ satisfies the unexpected cone property $\CC(a)$.

Next we turn to $\CC(b)$. Now we are assuming $a \geq 3$.

Since we have already taken care of the case $a=b$, we will assume here that $ 3 \leq a<b$. A point of multiplicity $b$ should impose $\binom{b+2}{3}$ conditions on surfaces of degree $b$, and we want to show that 
\[
\dim [I_Z \cap I_P^n]_b > \dim [I_Z]_b - \binom{b+2}{3}.
\]

Let us compute the Hilbert function of $Z$. As before, let $C$ be the union of the lines $L_1,\dots,L_a$, and note $[I_Z]_k = [I_C]_k$ for all $k \leq b-1$. We begin by understanding the Hilbert function of $C$. Consider the exact sequence
\begin{equation} \label{exact seq}
0 \rightarrow [I_C]_k \rightarrow [R]_k \rightarrow H^0(\mathcal O_C(k)) \rightarrow H^1(\mathcal I_C(k)) \rightarrow 0.
\end{equation}
The penultimate term has dimension $a(k+1)$, so to understand the Hilbert function of $C$ we have to understand $h^1(\mathcal I_C(k))$, and in particular understand when it is zero.

It is easy to see that $C$ is linked in a complete intersection of type $(2,a)$ to another union, $C'$, of $m$ skew lines. (If $a \geq 3$ there is a unique quadric containing $C$, and the lines of $C'$ are on the other ruling of the quadric). Clearly from the  sequence (\ref{exact seq}) we have $h^1(\mathcal I_C(k)) = 0$ for $k < 0$, and it is equal to $a-1$ for $k=0$. By the invariance (up to shifts and duals) of the Rao module $\bigoplus_{t \in \mathbb Z} H^1(\mathcal I_C(t))$ (see for instance \cite{migliorebook} Theorem 5.3.1) we know that the last degree for which $h^1(\mathcal I_C(k))$ is non-zero is 
\[
k = 0 +2 + a - 4 = a-2
\]
and furthermore $h^1(\mathcal I_C(a-2)) = a-1$.
(The $2$ and the $a$ are because those are the degrees of the surfaces participating in the link, and the $4$ is $3+1$ since we are working in $\mathbb P^3$). Furthermore, notice that a surface of degree $<a$ contains $C$ if and only if it is a multiple of the defining equation of the quadric. Finally, recall that the Hilbert function of a quadric surface is given by $h_Q(k) = (k+1)^2$ for all $k \geq 0$.

Using (\ref{exact seq}) and the above computation we obtain 
\[
h_C(k) = \dim [R/I_C]_k = 
\left \{
\begin{array}{ll}
 (k+1)^2 &   \hbox{if } k \leq a-1 \\ \\
a (k+1) & \hbox{if } k \geq a
\end{array}
\right.
\]
Consequently $h_Z(b-1) = ab$, i.e. the Hilbert function of $Z$ already stabilizes in degree $b-1$. We obtain
$
\dim [I_Z]_b = \binom{b+3}{3} - ab
$
so 
\[
\dim [I_Z]_b - \binom{b+2}{3} = \binom{b+2}{2} - ab.
\]

Now we use the fact that $Z$ satisfies $\CC(a)$. If $F$ is a defining form for the cone over $C$, then $FG$ has degree $b$ and multiplicity $b$ at the general point $P$, for all $G \in I_P^{b-a}$. But furthermore, the cone over the $b$ skew lines is not of this form. Thus
\[
\dim [I_Z \cap I_P^b]_b > \dim [I_P^{b-a}]_{b-a} = \binom{b-a+2}{2}.
\]
Combining, a calculation shows that $\binom{b-a+2}{2} \geq \binom{b+2}{2} - ab$ as long as $a \geq 3$, so
\[
\dim [I_Z \cap I_P^b ]_b > \binom{b-a+2}{2} \geq \binom{b+2}{2} -ab = \dim [I_Z]_b - \binom{b+2}{3}
\]
so $Z$ satisfies $\CC(b)$.
\end{proof}

\begin{remark}
In the context of Theorem \ref{gridex}, if $a=b = 2$ it is obvious that $Z$ does not satisfy $\CC(2)$ since essentially we have the union of a double point and four general points. If $a = 2$ and $b \geq 3$, one can check  small values of $b$ on the computer to see that such a grid does not satisfy $\CC(b)$.  While the details may be tedious and we have not written one down, we expect that a theoretical proof for all $b$ would boil down to showing that the cones of the form $FG$ (as in the proof) together with the cone over the $b$ skew lines form a basis for $[I_Z \cap I_P^b]_b$.

In particular, a $(2,b)$ grid shows that it is possible for a set of $ab$ points to have a general projection that is a complete intersection, but not satisfy both properties $\CC(a)$ and $\CC(b)$. We believe that this should be a very rare occurrence, however.
\end{remark}


\section{The property $\CC(2)$} \label{C(2) section}

In this section we consider and classify non-degenerate sets $Z$ of points in $\pp^3$ which satisfy the property $\CC(2)$,
i.e. sets with an unexpected quadric cone.

\begin{remark} 
If $Z$ satisfies the property $\CC(2)$, then $\ell(Z)\geq 6$.

Indeed, assume $\ell(Z)\leq 5$. Then the property $\CC(2)$ implies that a general projection of $Z$ to a plane $H$ fails to impose independent conditions on conics, by Proposition \ref{indep cond}. The only way in which a set $W$ of at most $5$ points in a plane $H$ does not impose
independent conditions on conics is that $W$ contains 4 aligned points.

Thus, for a general choice of $P\in\pp^3$, the set $Z$ contains a subset $Z'$ of $4$ points lying in a plane passing through $P$.
By continuity, the subset $Z'$ of $Z$ is fixed, as $P$ moves. Thus, if $P,Q$ are general, then $Z'$ lies in the intersection
of two different planes, passing through $P$ and $Q$ respectively. It follows that at least $4$ points of $Z$ are contained in a line.
Thus $Z$ lies in a plane, a contradiction.   
\end{remark}

Next, we prove that all sets $Z\subset \pp^3$ which satisfy the property $\CC(2)$ lie on two skew lines.

We could use an argument similar to the arguments that we will use to describe set satisfying $\CC(3)$, in the next section.
Instead, we prefer to prove the claim by a description of local deformations of cones. 

\begin{remark} Let $\pp^N$ be the space which parameterizes surfaces of fixed degree $d$ in $\pp^3$. 
Let $Z$ be any subvariety of $\pp^3$.

Let $W_Z$ be the subvariety of singular surfaces containing $Z$. Let $S \in W_Z$ be a point representing a surface with only one singular 
point, which is an isolated double point, at $P$. Then the (Zariski) tangent space to $W_Z$ at the point $S$ is represented by the quotient of $[I]_d$ mod an equation 
of $S$, where $I$ is the ideal of $Z \cup P$. (This is the classical \emph{infinitesimal Bertini's principle},  see e.g. Section 2 of \cite{CC02}.)
\end{remark}

\begin{proposition} \label{CC2} A non-degenerate set $Z$ of $\ell(Z)\geq 6$ points satisfies property $\CC(2)$ if and only if it is contained in a pair
of skew lines $L_1,L_2$, with each line $L_i$ containing at least $3$ points of~$Z$. 
\end{proposition}
\begin{proof} 
If $Z$ consists of $\ell(Z) \geq 6$ points on a pair of skew lines, with at least 3 points on each line, then by Proposition \ref{cone in P3} and Remark \ref{pts-curves}, $Z$ satisfies $\CC(2)$.

Conversely, assume that $Z$ satisfies property $\CC(2)$, so there is an unexpected quadric cone containing $Z$ with vertex (of multiplicity 2) at a general point $P$. We first claim that there is no {\em irreducible} such quadric for general $P$. Suppose, by way of contradiction, that for a general $P$ there is an irreducible quadric cone containing $Z$ and having vertex $P$. Note first that such a quadric has to be unique, since the intersection of two such is a union of four lines containing the points of $Z$, which is not possible when $P$ is general.  So suppose that for a general $P$ there is a unique irreducible quadric cone containing $Z$ and having vertex $P$. 
Since the vertices of these irreducible quadric cones dominate $\pp^3$
by a dimension count, then there is a $3$-dimensional family $W'$ of irreducible quadric cones containing $Z$, with a moving isolated double point.  By the previous remark, the tangent space to $W'$ at a general point $S\in W'$
is represented by the space $[I']_2$ mod an equation of $S$, where $I'$ is the ideal of $Z'=Z\cup \{P\}$, $P$ being the vertex of $S$, 
i.e. a general point. 
So $I'_2$ has dimension $4$. As $P$ is general, this is only possible if the ideal of $Z$, in degree $2$, has dimension 
$5$. 

Thus the Hilbert function $h_Z$ of $Z$ starts with $h_Z(0)=1$, $h_Z(1)=4$, $h_Z(2)=5$, i.e. the first difference $\Delta h_Z$ begins $(1,3,1,\dots)$. Since $\ell(Z) \geq 6$, it follows by Theorem \ref{BGM} that $\Delta h_Z(i) = 1$ for $2 \leq i \leq \ell(Z) - 3$ and hence
$h_Z(i)=\min \{\ell(Z), i+3\}$ for all $i\geq 2$, and moreover all the points of $Z$, except two of
them, are contained in a line. Thus, there exists a reducible quadric cone through $Z$ with vertex at a general point, a contradiction.

It follows that when $Z$ satisfies property $\CC(2)$, for a general choice of $P\in\pp^3$ $Z$ lies in a pair of planes $\pi,\pi'$ meeting at  $P$. These two planes are distinct, since $Z$ is reduced and does not lie in a plane. Moving $P$ generically,
the two planes move in two non-trivial families. The subsets of $Z$ contained in any elements of the two families of planes are fixed,
by continuity. It follows that $Z$ splits in a union $Z=Z_1\cup Z_2$ where each $Z_i$ lies in the base locus of a non-trivial
family of planes, i.e. it lies in a line $L_i$. Since $Z$ does not lie in a plane, the two lines $L_1,L_2$ are skew,
and each of them cannot contain more than $\ell(Z)-2$ points of $Z$.

Finally, if one of the lines contains only two points of $Z$, we have that $Z \cup 2P$ imposes  $4 + (2+1) + 2 = 9$ conditions so the quadric cone is not unexpected, contradicting the assumption $\CC(2)$. 
\end{proof}

\begin{remark}
Thanks to Proposition \ref{CC2} we see immediately that a set of points can satisfy $\CC(2)$ but not have a general projection to a plane that is a complete intersection. For instance, we can take 5 points on $L_1$ and 3 points on $L_2$.
\end{remark}

We end this section with a proposition on the case of points contained in too many quadric cones, which will be useful in the sequel.

\begin{proposition}\label{2pen} Let $Z$ be a set of at least $5$ points in $\pp^3$, such that for a general $P\in\pp^3$ $Z$ lies
in a pencil of quadric cones, with vertex at $P$. Then all the points of $Z$, except at most one of them, lie in a line.
In particular, $Z$ lies in a plane.
\end{proposition}
\begin{proof} For $P\in\pp^3$ general, the projection of $Z$ to a general plane lies in a pencil of conics. It is straightforward that
if a set of at least $5$ points in a plane lie in a pencil of conics, then all the points, except at most one of them, are collinear.
Thus, there exists a subset $Z'\subset Z$, of length at least $\ell(Z)-1$, which lies in a plane $\pi_P$ through $P$. Moving $P$ to a general
point $P'$ we get a new plane $\pi_P'$ through $P'$ which contains a subset  of at least $\ell(Z)-1$ points of $Z$. By continuity,
this subset is necessarily $Z'$. Thus $Z'$ is aligned.
\end{proof}


\section{On the property $\CC(3)$} \label{C(3) section}

In this section we analyze the structure of sets of points which satisfy the property $\CC(3)$. We will be able to do that only up to 
$\ell(Z)=9$. Our main goal is to understand, for a set of 9 points $Z \subset \mathbb P^3$, how the following properties are related:

\begin{itemize}

\item[(A)] $Z$ satisfies the unexpected cubic property $\CC(3)$;

\item[(B)] The general projection of $Z$ to a plane is a complete intersection of two cubic curves;

\item[(C)] $Z$ is a $(3,3)$-grid.

\end{itemize}

\begin{remark} \label{1st remarks}
We have some immediate observations. 

\begin{itemize}

\item[(a)] The implication (C) $\Rightarrow$ (A) is  a special case of Theorem \ref{gridex}. 

\item[(b)] If $Z$ is a $(3,3)$-grid it is clear that a general projection of $Z$ to a plane is  the complete intersection of the projection of the lines $L_i$ and the projection of the lines  $L_j'$, which shows that (C) $\Rightarrow$ (B). 

\item[(c)] Finally, it is also clear that if $Z$ imposes independent conditions on cubic surfaces then $\dim [I_Z]_3 - \binom{3+3-1}{3} = 1$ so (B) $\Rightarrow$ (A) in this situation.
\end{itemize}
\end{remark}

First, let us dispose of the case $\ell(Z)<9$.

Of course, since we already discussed property $\CC(2)$ in the previous section, \emph{we will always assume here that $Z$ does not satisfy
property $\CC(2)$}.

\begin{remark}\label{lZ7} 
If $Z$ satisfies property $\CC(3)$ and does not satisfy property $\CC(2)$, we claim that $\ell(Z)\geq 8$.
Indeed, suppose that $\ell(Z) \leq 7$. By Proposition~\ref{CC2}, the assumption that $Z$ does not satisfy $\CC(2)$ means that $Z$ does not consist of 3 points on one line together with either 3 or 4 points on a second line, disjoint from the first. 

Now assume that $Z$ satisfies $\CC(3)$ and consider the  projection $\pi_P(Z)$ of $Z$ from a general point to a general plane. By Proposition \ref{indep cond}, $\pi_P (Z)$ does not impose independent conditions on cubics. In particular, $\ell(\pi_P(Z)) = \ell(Z) \geq 5$. The following are the only possible $h$-vectors for $Z$:
{\small
\[
(1,1,1,1,1), \ \ (1,1,1,1,1,1), \ \ (1,1,1,1,1,1,1), \ \ (1,2,1,1,1), \ \ (1,2,1,1,1,1), \ \ (1,3,1,1,1).
\] }
In particular, by Theorem \ref{BGM}, $Z$ must contain a subset of at least five points on a line, plus at most two other points (possibly also on the same line). Of these, all but the last lie in a plane so by Lemma \ref{planar pts} they do not admit unexpected cones, contradicting $\CC(3)$. 

The last possibility above consists of a set of exactly five collinear points plus two more points, neither on the line, and this also describes the general projection to a plane. Then both $\pi_P(Z)$ and $Z$ fail by one to impose independent conditions on cubics. We have
\[
\dim [I_Z]_3 - 10 = 14 - 10 = 4
\]
and 
\[
\dim [I_Z \cap I_P^3]_3 = \dim [I_{\pi_P(Z)}]_3 = 4
\]
so there is no unexpected cubic cone here either.
\end{remark}

\begin{remark} \label{C3notC2}
We want to show that there do not exist 8 points satisfying property $\CC(3)$ but not property $\CC(2)$.
Assume that $\ell(Z)=8$, and $Z$ satisfies property $\CC(3)$ and it does not satisfy property $\CC(2)$. We claim that there would have to be a line $L$ which contains
exactly $5$ points of $Z$, and the remaining $3$ points of $Z$ span a plane (not containing $L$).

First assume that the line $L$ exists. Since $Z$ itself does not span a plane, if the $3$ points of $Z$ missed by $L$ are collinear, then  $Z$ lies in two skew lines. Thus, by Proposition \ref{CC2}, $Z$ satisfies property $\CC(2)$, a contradiction.

Now assume that $Z$ satisfies property $\CC(3)$ but does not satisfy property $\CC(2)$. 
Let $P$ be a general point in $\mathbb P^3$ and let $H$ be a general plane. Let $Z' = \pi_P(Z)$, the projection to $H$ from $P$. 
By Proposition \ref{indep cond}, $Z'$ does not impose independent conditions on cubics. This means that the $h$-vector of $Z'$ must be either  $(1,2,3,1,1)$ or $(1,2,2,2,1)$.

Suppose first that the $h$-vector of $Z'$ is $(1,2,3,1,1)$. 
By {Theorem \ref{BGM},} this implies that $5$ points of $Z'$  are contained in a line $L'$ and the remaining three are non-collinear and none of them lies on $L$.  Thus, a subset $Z_0$ of $5$ points of $Z$ lies in the plane $\pi$ spanned by $L',P$. Moving $P$, we find that for a general $Q$ a subset of $5$ points of $Z$ lies in a plane $\pi'$ passing through $Q$. By continuity, this subset is again $Z_0$. Thus $Z_0$ is contained in the intersection of two distinct planes, i.e. it lies in a line $L$. The remaining three points are non-collinear and none lies on the line of $Z_0$ (since the general projection contains five and not six collinear points).
But because one computes $\dim [I_Z]_3 = 13$ (and not 12) and $\dim [I_{Z'}]_3 = 3$, we see that the cubic cones are not unexpected, contradicting $\CC(3)$. 

It remains to consider the case where the $h$-vector of $Z'$ is $(1,2,2,2,1)$. If $Z'$ contains five collinear points then the remaining three points must lie on a line (since $Z'$ lies on a conic). This means, as above, that $Z$ contains 5 points on one line and 3 on a disjoint line, which means $Z$ satisfies $\CC(2)$. Contradiction. 

So $Z'$ does not contain five collinear points. Then it is not hard to see that $Z'$ is the complete intersection of a plane curve of degree 2 (possibly two lines) and a plane curve of degree 4. If the conic consists of two lines then $Z'$ is the union of two sets of four collinear points, and since the projection is general, $Z$ is the union of two sets of four points on each of two skew lines, so by Proposition \ref{CC2} it satisfies $\CC(2)$. 

We have only to show that the complete intersection of a smooth conic and a plane curve of degree 4 cannot occur in our situation. 
We have
\[
3 = \dim [I_{Z'}]_3 = \dim [I_Z \cap I_P^3]_3 .
\]
Since $Z$ satisfies $\CC(3)$, we must have $\dim [I_Z]_3 - 10 \leq 2$, so $\dim [I_Z]_3 \leq 12$. Since $h_Z(3) \leq 8$, we get $\dim [I_Z]_3 = 12$ and $Z$ imposes independent conditions on cubics. On the other hand, the general projection of $Z$ lies on a conic, and the cone over this conic contains $Z$, so moving $P$ we see that $\dim [I_Z]_2 \geq 4$. These two facts  mean that the $h$-vector of $Z$ is  $(1,3,2,2)$. This forces $Z$ to lie on a curve of degree 2, and since the projection lies on a smooth conic, $Z$ must itself lie on a conic. Then by Lemma \ref{planar pts} $Z$ does not have $\CC(3)$. Contradiction. 
\end{remark}

We  define an unexpected cone property $\CC(d)$ to be {\it improper} if the following two conditions both occur:

\begin{itemize}
\item[(a)]  the expected number of cones with general vertex $P$ passing through $Z$ is $0$, and

\item[(b)] $Z$ lies in a non-degenerate space curve of degree $d$. 

\end{itemize}

\noindent In this case, clearly, a general projection of $Z$ lies in a plane curve of degree $d$, so the corresponding cone with vertex $P$ is unexpected and  so
$Z$ satisfies $\CC(d)$ trivially. 

The previous discussion proves that the unexpected cone property $\CC(2)$ is always improper. 

\subsection{The $3\times 3$ case}

From now on, we take a set $Z$ of $9$ points in $\pp^3$ and we assume that $Z$ satisfies condition $\CC(3)$.
 It turns out that this forces condition $\CC(2)$ to fail (see Remark~\ref{C2 C3 incomp}).

\begin{lemma} \label{indepcond}
 If $Z$ fails to impose independent conditions on cubic surfaces then $Z$ fails to satisfy $\CC(3)$, and in addition the general projection of $Z$ to a plane fails to be a complete intersection.
 \end{lemma}
 
\begin{proof}
Assume that our set of 9 points, $Z$, does not impose independent conditions on cubics.  We will continue to use the language of $h$-vectors, which for a zero-dimensional scheme $Z$ is the first difference of the Hilbert function of $Z$. The results of \cite{BGM94} that we use are in this language.

Indeed, if $Z$ does not impose independent conditions on cubics then the possible $h$-vectors for $Z$ are
\[
(1,3,3,1,1), (1,3,2,2,1), (1,3,2,1,1,1),  \hbox{ or } \ (1,3,1,1,1,1,1)
\]
Let $P, H, Z'$ be as  in Remark \ref{C3notC2}. Since $\CC(2)$ does not hold, by Proposition \ref{CC2} $Z$ is not the union of 5 points on one line and 4 on a second line, nor is it the union of 6 points on a line and 3 on a second line. 

If the $h$-vector of $Z$ is (1,3,3,1,1) then $Z$ contains five points on a line and four points not all on a line. This forces the $h$-vector of $Z'$ to be $(1,2,3,2,1)$. Then there is a pencil of cubics through $Z'$. On the other hand, $\dim [I_Z]_3 - 10 = 12 - 10 = 2$ so there are no unexpected cones and $\CC(3)$ does not hold. The collinear points preclude the projection from being the complete intersection of cubics.

If the $h$-vector of $Z$ is (1,3,2,2,1) then at least 8 points of $Z$ lie on a curve of degree~$2$ by \cite{BGM94} Theorem 3.6 (and an easy calculation). If that curve consists of two skew lines $L_1 \cup L_2$ then the length of the $h$-vector forces one of the following without loss of generality:

\begin{itemize}

\item[(i)] 5 points on $L_1$, 4 points on $L_2$;

\item[(ii)] 5 points on $L_1$, 3 points on $L_2$, one additional point;

\item[(iii)] 4 points on $L_1$, 4 points on $L_2$, one additional point;

\end{itemize}

\noindent Case (i) is ruled out since we assume $\CC(2)$ does not hold. Cases (ii) and (iii) do not have the desired $h$-vector. So in fact $Z$ consists of  $8$ points on a plane curve of degree $2$ together with one additional point. Then an analysis similar to that of the previous case continues to hold.

If the $h$-vector of $Z$ is (1,3,2,1,1,1) then $Z$ contains $6$ points on a line plus $3$ non-collinear points. The same analysis shows $\CC(3)$ does not hold and that the general projection is not a complete intersection.

If the $h$-vector of $Z$ is (1,3,1,1,1,1,1) then $Z$ consists of $7$ points on a line plus $2$ points. Computing the Hilbert function, we see that $\dim [I_Z]_3 = 20-6 = 14$ so a general point of multiplicity $3$ is expected to impose $\binom{3-1+3}{3} = 10$ conditions and we expect $\dim [I_Z \cap I_P^3]_3 = 4$. On the other hand, the projection $\pi_P(Z)$ from a general point $P$ has $h$-vector $(1,2,2,1,1,1,1)$ so $\dim [I_{\pi_P(Z)}]_3 = 4$, and so there is no {\em unexpected} cubic cone, i.e. $Z$ does not satisfy $\CC(3)$.
\end{proof}

\smallskip

Thus, we know that our set of 9 points, $Z$, lies neither in a plane, nor on two skew lines. Since $20-\ell(Z)=11$ and $Z$ imposes independent conditions, we expect to find a single
cubic cone through $Z$, with vertex at a general point $P\in \pp^3$. Thus $Z$ does not satisfy property $\CC(3)$ if and only if
there is (at least) a pencil of cubic cones through $Z$ with vertex at a general point $P\in\pp^3$.

\begin{remark} \label{C2 C3 incomp}
Proposition \ref{CC2} says (in particular) that a non-degenerate set of 9 points $Z$ in $\mathbb P^3$ satisfies property $\CC(2)$ if and only if it is contained in a pair of skew lines, each containing at least 3 points of $Z$. Note that any such set necessarily fails to impose independent conditions on cubic surfaces since one line must contain at least 5 points. Lemma \ref{indepcond} says that if $Z$ satisfies $\CC(3)$ then it {\em does} impose independent conditions on cubic surfaces. 
Thus a set of 9 points in $\mathbb P^3$ cannot simultaneously satisfy $\CC(2)$ and $\CC(3)$. 
In particular,  for sets of 9 points, $\CC(3)$ implies that $\CC(2)$ does {\em not} hold. 
Contrast this with Remark \ref{C3notC2}, where it was shown that for 8 points $\CC(3)$ {\em forces} $\CC(2)$.  
\end{remark}

\begin{proposition} \label{reducible}
Assume that $Z \subset \mathbb P^3$ consists of 9 points satisfying $\CC(3)$ and hence not $\CC(2)$. If the general cubic cone through $Z$ is reducible then $Z$ consists of 4 points on one line, 4 points on a second (disjoint) line, plus one more point.
\end{proposition}

\begin{proof}
We can assume, by Lemma \ref{indepcond}, that $Z$ imposes independent conditions on cubic surfaces, so $\CC(3)$ means that we have a pencil (or more) of cubics containing $Z$. In particular, $Z$ does not contain a subset of $\geq 5$ points on a line, or 8 points on a plane curve of degree 2, and $Z$ itself cannot consist of 9 points on a set of two skew lines in $\mathbb P^3$.

Suppose that the general cubic cone through $Z$ is reducible. 
Any reducible cubic cone contains a plane passing through the vertex. Fix  $P\in\pp^3$ general, fix a general
cubic cone with vertex at $P$ containing $Z$,  and call $\Pi_P$ a plane through $P$ contained in the cone. $\Pi_P$ contains a subset $Z_0$ of $Z$. 
If we move $P$ generically to a point $P'$, the plane $\Pi_P$ moves to another plane $\Pi_P'$ through $P'$ which,
by continuity, still contains $Z_0$. Thus $Z_0$ lies on a line.
If $\ell(Z_0)\geq 5$,  we have a contradiction (there are at least 5 points on a line), so assume $\ell(Z_0)\leq 4$.

If $2 \leq \ell(Z_0) \leq 4$ then the plane $\Pi_P$ is fixed, and there is at least a pencil of quadric cones through $Z' = Z \backslash Z_0$ with vertex $P$, and we have $5 \leq \ell(Z') \leq 7$. This means we have at least a pencil of conics through $\pi_P (Z')$, so $\pi_P(Z')$ must consist of $\ell(Z') -1$ points on a line plus one other point, or else $\ell(Z')$ points on a line. Since $\ell(Z') \geq 5$, but $Z$ does not contain 5 points on a line, we must have $\ell(Z') = 5$ and $\pi_P(Z')$ is a set of 4 points on a line plus one other point. Since this is a general projection, we conclude that $Z$ consists of $4$ points on one line (namely $Z_0$), $4$ points on a second line disjoint from the first, plus one more point. 

Suppose $\ell(Z_0) = 1$. Then there is a pencil of planes containing $Z_0$ and the general point $P$. The general projection of $Z'$ to a general plane must then be a set of $8$ points on a conic. Since $Z$ does not contain a subset of $8$ points on a conic, we must again have that $Z'$ consists of two sets of four collinear points plus one point ($Z_0$). 

Finally, suppose $Z_0$ is empty, i.e. all the points of $Z$ lie on a quadric cone and the plane is any plane through the general point. This forces $Z$ to lie on two skew lines (since it cannot lie on a conic), violating the fact that $Z$ does not contain a subset of $5$ collinear points.
\end{proof}

\begin{remark} \label{assume indep}
Thanks to Proposition \ref{reducible}, it is now enough to consider the case that a general cubic cone containing $Z$ is irreducible. It follows that the projection of $Z$ from a general point $P$ 
to a general plane is a set of $9$ points which are contained in a pencil of cubic curves, whose general element is irreducible. Such a set
is necessarily complete intersection of two cubic curves.

Thus, {\it we assume from now on that the general projection of
$Z$ to a plane $H$ is a complete intersection of two cubic curves in $H$.}

Furthermore, by Lemma \ref{indepcond} {\em we can also assume that $Z$ imposes independent conditions on cubic surfaces in $\mathbb P^3$.}
\end{remark}

\begin{remark} \label{semicont} We are assuming that a general projection of $Z$ lies in a pencil of cubic curves, giving a complete
intersection. This is no longer obvious for \emph{special} projections of $Z$. In any case, by semicontinuity,
even a special projection of $Z$ is contained in a family of cubic curves of (projective) dimension at least $1$. 
\end{remark}

\begin{remark} \label{no4all}
Assume that $Z$ contains $4$ collinear points. Then a general projection of $Z$ cannot be a complete
intersection of two cubic curves. Namely, the projection too has at least $4$ points on a line $L$, so that any cubic curve
containing the projection will contain~$L$.

Similarly, assume that $Z$ contains $7$ points on an irreducible conic. Then a general projection of $Z$ cannot be a complete
intersection of two cubic curves. Namely, the projection too has at least $8$ points on an irreducible conic $\Gamma$, so that any cubic curve
containing the projetion will contain $\Gamma$.
\end{remark}

\begin{remark}
We have seen in Remark \ref{assume indep} that we may assume that $Z$ imposes independent conditions on cubics. Then, as we have seen in Remark \ref{1st remarks}, the property that a general projection is a complete intersection of two cubics implies the unexpected cone property $\CC(3)$.
\end{remark}

We know from Theorem \ref{gridex} that any $(3,3)$-grid satisfies property $\CC(3)$. We have also seen in Remark \ref{C2 C3 incomp}  that $Z$ cannot simultaneously have $\CC(3)$ and $\CC(2)$. The aim of the rest of this section is to prove that sets $Z$ of $9$ points in $\pp^3$ which satisfy property $\CC(3)$  and are not of the type described in Proposition \ref{reducible}   are $(3,3)$-grids. In particular, we can assume that the general projection to a plane is the complete intersection of two cubics.

\begin{remark}\label{no6conic}
Assume  that $Z$ is a set of $9$ point which  satisfies property $\CC(3)$, and assume that  no $3$ points of $Z$ are collinear.
Then, a general projection of $Z$ cannot contain three collinear points. Otherwise, we get that for a general
 $P$ there are three points of $Z$ which span a plane, together with $P$. This is possible only if three points
 of $Z$ are collinear.  
 
 Similarly, a general projection of $Z$ cannot contain $6$ points on a conic. Indeed otherwise, since the general
 projection does not contain three collinear points, then the conic containing six points of a general projection
 is unique and irreducible. Since the projection is a complete intersection of two cubics, by Cayley-Bacharach \cite{EGH}
 the remaining three points of the projection of $Z$ are collinear, a contradiction. 
\end{remark}

\begin{theorem} \label{proj CI33} Let $Z\subset \pp^3$ be a non-degenerate set of $9$ points 
such that the general
projection of $Z$ to a plane is a complete intersection of two cubic curves. Then $Z$ is a $(3,3)$-grid.
\end{theorem}

\begin{proof} 

In this proof, on more than one occasion we will refer to the fact that six points on a conic must be linked, in a complete intersection of cubics in the plane, to three collinear points. This is well known, and is a variant of the Cayley-Bacharach theorem \cite{EGH}.

We know from Lemma \ref{indepcond} that if $Z$ projects to the complete intersection of two cubics in the plane then $Z$ imposes independent conditions on cubic surfaces. Then Remark \ref{1st remarks} (c) gives us that $Z$ satisfies the unexpected cone property $\CC(3)$. Thus we can also use this fact about $Z$. Notice also that certainly no four points of $Z$ are collinear, since then the projection would have this property and then could not be a complete intersection of cubics. 

\medskip

\noindent 
{\it Claim:} Necessarily  $3$ points of $Z$ are collinear.

\medskip 

Let us argue by contradiction. 

{ \bf Assume first that no $3$ points of $Z$ are collinear and no $4$ points of $Z$ are coplanar (so that $Z$ is in 
 linear general position).} 
 
 In the different parts of this proof we will consider different choices of a subset $Y$ of $Z$. Similarly, in different parts of the proof we will consider different centers of projection, $P$, to a general plane.
 
 We start with a subset $Y$ of  exactly $6$ points of $Z$.
 It follows from  Castelnuovo's Lemma  (see for instance \cite{harris} Theorem 1.18) that $Y$ is contained in a
(unique) irreducible rational cubic curve $C$. It may or may not happen that other points of $Z$ also lie on $C$. Let $P$ be a general point of $C$. Clearly $P$ avoids the secant lines joining points of $Z$, and recalling that $\ell(Y) = 6$, we also get that $P$ also avoids the plane of $Z' := Z \backslash Y$. Thus the projection $\pi_P$ from $P$ to a general plane gives at least $6$ points on the conic $\pi_P(C)$, {\it and the remaining points are in general linear position in the plane.}

Recall that by Remark \ref{semicont}, for any projection $\pi_P(Z)$ must lie on at least a pencil of cubics, whether they define a complete intersection or not. Now consider a subset $Y$ of {\it at least} $6$ points, lying on a twisted cubic. We consider cases.

\begin{enumerate}

\item $\ell(Y) = 6$. Then $\pi_P(Z')$ is a set of three non-collinear points. But $6$ points on a conic are linked by two cubics to $3$ collinear points. Contradiction.

\item $\ell(Y) = 7$. The cubics through $\pi_P(Z)$ all have the conic $\pi_P(C)$ as a component by Bezout's theorem. But we have two points off of $\pi_P(C)$, so there cannot be a pencil of cubics through $\pi_P(Z)$. Contradiction.

\item $\ell(Y) = 8$, so $Z$ consists of $8$ points on a twisted cubic plus one point off the twisted cubic. Now instead, as above, let $P$ be a general point in $\mathbb P^3$ and $\pi_P$ the projection to a general plane. Then $\pi_P(Z)$ contains $8$ points on the cubic curve $\pi_P(C)$ and one point off this curve. But if $\pi_P(Z)$ were a complete intersection of cubics, any cubic curve through $8$ of the points has to pass through the ninth, by the classical Cayley-Bacharach theorem. Contradiction.

\item $\ell(Y) = 9$, so $Z$ is a set of $9$ distinct points on a twisted cubic curve $C$. Let us label the $9$ points of $Z$ as $Q_1,Q_2,Q_3,R_1,R_2,R_3, S_1,S_2,S_3$ and by $\Pi_Q, \Pi_R, \Pi_S$ the corresponding planes spanned by the obvious triples of points. Relabelling if necessary, we can assume that $\Pi_Q, \Pi_R, \Pi_S$ do not share a line.
Let $\lambda = \Pi_Q \cap \Pi_R$ be the line of intersection of the first two planes. Let $P$ be a general point of $\lambda$ (thus $P$ is not on $\Pi_S$) and $\pi_P$ the projection from $P$ to a general plane. Then $\pi_P(Q_1 \cup Q_2 \cup Q_3)$ and $\pi_P(R_1 \cup R_2 \cup R_3)$ are  each sets of three collinear points, while $\pi_P(S_1 \cup S_2 \cup S_3)$ are three non-collinear points. But then $\pi_P(Q_1 \cup Q_2 \cup Q_3 \cup R_1 \cup R_2 \cup R_3)$ is a set of six points on a conic (the union of two lines). We know that $\pi_P(Z)$ lies on an irreducible cubic curve, namely $\pi_P (C)$. If there is another cubic containing $\pi_P(Z)$, it then defines a complete intersection of cubics. But  in such a complete intersection, the residual of six points on a conic is a set of three collinear points, contradicting the fact that $\pi_P(S_1 \cup S_2 \cup S_3)$ are three non-collinear points. Hence $\pi_P(Z)$ lies on a unique cubic curve, contradicting Remark~\ref{semicont}.

\end{enumerate}

{\bf With this we can now assume that $Z$ is not in linear general position}, but we still assume that no three points are on a line. Thus there are at least four points on a plane, $H_0$. 
Set  $Y_1 = H_0 \cap Z$ and call $\ell $ the cardinality of $Y_1$, $4 \leq \ell\leq 8$. We may assume that $Y_1 = \{ P_1, \dots , P_\ell \}$, $Y_2 = \{ P_{\ell+1},\dots,P_9\}$ and $Z=Y_1\cup Y_2$.

Again we consider cases.
 
\begin{enumerate}

\item Assume  $\ell = 4$. Let $P$ be a general point of $H_0$ and let $\pi_P$ be the projection from $P$ to a general plane $H$. Notice that $\pi_P(Y_1)$ is a set of $\ell = 4$ collinear points, so any cubic containing $\pi_P(Z)$ has the line spanned by $\pi_P(Y_1)$ as a component. Hence the five remaining points $\pi_P(Y_2)$ lie on a pencil of conics in $H$. 
Then $\pi_P (Y_2)$ contains at least 4 points on a line, which is impossible since $Z$ does not contain three collinear points and $P$ is general in $H_0$. 

\item If $\ell = 5$ we consider two cases. Suppose the four points of $Y_2$ also lie on a plane. Then in effect this was handled in the first case. So we can assume that $Y_2$ does not lie on a plane. Let $C$ be the conic containing $Y_1$ and let  $S$ be the quadric cone over $C$ with vertex $P_6$. Let $H$ be the plane spanned by $P_7, P_8, P_9$. Let $D = S \cap H$. Since $Y_2$ does not lie on a plane, $D$ does not pass through the vertex $P_6$ of $S$, so $D$ is a smooth conic. Let $Q$ be a general point of $S$ not on $D$. Let $\pi_Q$ be the projection from $Q$. Then $\pi_Q(P_6)$ lies on $C$ (since $Q \in S$) but the points of $\pi_Q(P_7 \cup P_8 \cup P_9)$ are not collinear (since $Q \notin H$). Contradiction.

\item If $\ell = 6$ then $Y_1$ is a set of $6$ points in linear general position in $H_0$, not on a conic (Remark \ref{no6conic}) and $Y_2$ is a set of $3$ points not lying on $H_0$ and not collinear. Let $H'$ be the plane spanned by $Y_2$ and let $Q$ be a general point on $H'$. Let $\pi_Q$ be the projection from $Q$ to $H_0$. Note that $\pi_Q(Y_2)$ is a set of three  collinear points. In this construction the pencil of cubics containing $\pi_Q(Z)$ is a complete intersection, which links three collinear points to a set of six points not on a conic.  Contradiction.

\item If $\ell = 7, 8$, $Y_1$ is a set of 7 points on $H_0$ in linear general position, no $6$ on a conic. Then by considering the plane $H'$ spanned by $P_7, P_8, P_9$, the same argument as in the case $\ell = 6$ works.
\end{enumerate}

This ends the proof of the claim.  
 
So assume that $Y_1 = \{P_1,P_2,P_3 \}$ is a set of collinear points in $Z$, in the line $L$, and let $Y_2$ be the residual, $Y_2 = \{ P_4,\dots,P_9 \}$. Let $\pi_P$ be the projection from a general point $P$ to a general plane. Since $\pi_P(Z)$ is the complete intersection of two cubics, $Z$ (and hence $Y_2$) does not contain a set of four collinear points, and similarly $Y_2$ cannot contain a point of $L$. 
Thus $\pi_P(Y_1)$ is a set of 3 collinear points and $\pi_P(Y_2)$ is a set of six points on a conic (possibly reducible). 
Then $Y_2$ is contained in a quadric cone with vertex at a general point. 
Looking at the proof of Proposition \ref{CC2}, this is enough to conclude that $Y_2$ consists of six points lying on two  lines, $L', L''$. Since $Z$ does not have four points on a line, $Y_2$ has three points on $L'$ and three on $L''$.

Proposition \ref{CC2} assumed that the points are non-degenerate to conclude that the lines are skew, but in our situation we have to rule out the possibility that the two lines meet in a point. By interchanging the roles of the points, and recalling that $Z$ does not lie on a plane, the only possibility is that the three lines $L, L', L''$ meet in a point. Such a set of lines does not have any trisecant lines. Taking one point of $Z$ on each line we get a spanned plane $H_1$, and repeating we get a plane $H_2$. These have a line, $\lambda$, of intersection. Projecting from a general point on $\lambda$ to a general plane gives two points on each of two lines plus three non-collinear points, and such a set lies on a unique cubic curve. So we obtain that $L, L', L''$ are skew lines.
 
 Call $S$ the unique smooth quadric surface that contains
 the three lines. $L,L',L''$ determine one ruling of $S$. 
 
 Consider the plane $H$ spanned by $L,P_4$. We claim that $H$ meets $L''$ in a point of $Z$. Indeed the
 projection  $\pi_Q$ from a general point $Q\in H$ to a general plane contains four points on the line $\pi_Q(L)$. Thus any cubic  containing the projection must contain $\pi_Q(L)$. Since $\pi_Q(Z)$ sits in a pencil of cubics and the
 projection of the set $Z'=\{P_5,\dots,P_9\}$ cannot contain four collinear points, we get that $\pi_Q(L)$ must contain at least one more point of $Z$, i.e. some point of
 $Z'$ lies in $\pi_Q(L)$. This is impossible, for $Q$ general in $H$, for $\pi_Q(P_5),\pi_Q(P_6)$, so we get our claim. 
 
 Say that $P_7\in H$. Then the intersection of $H$ and $S$ contains the line $P_4,P_7$.
 
 Now, repeat the previous construction with the plane $H_2$ spanned by by $P_4$ and $L''$. We find that one point of $Z\cap L$,
 say $P_1$, is the intersection of $H$ and $L$. Since $H$ meets $S$ in $L''$ and the line $P_4,P_7$,
 it follows that $P_1,P_4,P_7$ are collinear, their line being one element of the
 ruling of $S$ opposite to the ruling determined by $L,L',L''$.
 
 Now, consider the plane $H_3$ spanned by $L,P_5$, repeat the construction. We see that, after renumbering,
the points $P_2,P_5,P_8$ are collinear, their line being one element of the
 ruling of $S$ opposite to the ruling determined by $L,L',L''$. If we consider now the plane $H_4$ spanned by $L,P_6$,
 we see that also $P_3,P_6,P_9$ are collinear, their line being one element of the
 ruling of $S$ opposite to the ruling determined by $L,L',L''$.
 
 In conclusion, $Z$ is a $(3,3)$-grid.
 \end{proof}


\section{Further geometric considerations} \label{geom section}

At the beginning of the previous section we mentioned that for a set $Z$ of $9$ points, the property $\CC(3)$ is  connected with the property that
a general projection of $V$ to a plane is a complete intersection (of two cubics).

In general, the relations between $\CC(d)$ and a \emph{projection to complete intersection} property are subtle. We begin by exploring this situation.

\begin{remark}
It is only a matter of computation to realize that, for any $d\geq 3$, if a set $Z$ of $d^2$ points in $\pp^3$, {\em imposing
independent conditions to surfaces of degree $d$}, has general projection which is
a complete intersection of type $(d,d)$, then it satisfies the unexpected cone property $\CC(d)$. 
\end{remark}

About the relation, more can be said. Indeed, for sets projecting to complete intersections the difference between the virtual and
the actual dimensions of the family of cones with general vertex is rather big.

\begin{definition}
The \emph{unexpected cone $d$-defect $\delta \CC(d)$} of $Z$ is the difference between the (virtual) expected dimension
of the family of cones with general vertex and its actual dimension, i.e.
 \[
 \delta \CC(d) = \dim [I_Z \cap I_P^d]_d - \dim [I_Z]_d + \binom{d+2}{3}.
 \]
\end{definition}

\smallskip

For instance, let us consider in detail what happens in the case $d=4$.
Sets $Z$ of $16$ points satisfying $\CC(4)$ and not $\CC(3)$, such that a general projection of $Z$ is a complete intersection
of type $(4,4)$ exist: just take $Z$ to be a $(4,4)$-grid. In fact, we will see in the appendix that there are also non-grid examples of such sets $Z$. 

\begin{proposition}
Let $Z$ be a set of $16$ points in $\mathbb P^3$ satisfying $\CC(4)$ but not $\CC(3)$. Then
$Z$ satisfies $\delta \CC(4) \leq 3$.
\end{proposition}
\begin{proof}
For a set $Z$ of $16$ points whose general projection is a complete intersection of type $(4,4)$ we have, for a general point $P$, that  $\dim [I_Z \cap I_P^4]_4 = 2$, and 
\[
\delta \CC(4) = 2 - [35 - h_Z(4)] + 20 = h_Z(4) - 13 \leq 3
\]
with equality if and only if $Z$ imposes independent conditions on quartics.

This assumed that the general projection is a complete intersection. In fact, we claim that the defect $\delta \CC(4)=3$ for a set of $16$ points is maximal, {\em among sets satisfying $\CC(4)$ and not $\CC(3)$.}
Indeed, assume that $Z$ has defect $\delta C(4)=4$ and it does not satisfy $\CC(3)$. This means that 
\[
4 =  \delta \CC(4) = \dim [I_Z \cap I_P^4]_4 - [35 - h_Z(4)] + 20 = \dim [I_Z \cap I_P^4]_4 + h_Z(4) - 15.
\]
Thus $\dim [I_Z \cap I_P^4]_4 = 19 - h_Z(4)$.

\medskip

\noindent \underline{Case 1}. If $Z$ imposes independent conditions on quartics, then $h_Z(4) = 16$, so we have
\begin{equation} \label{eq1}
\dim [I_Z \cap I_P^4]_4 = 3.
\end{equation} 
We first claim that the general projection cannot lie on a cubic curve. Indeed, if it did then we would have $\dim [I_Z \cap I_P^3]_3 \geq 1$. But $Z$ does not satisfy $\CC(3)$, so such a cubic surface is not unexpected, i.e.
\[
1 \leq \dim [I_Z \cap I_P^3]_3 = [20 - h_Z(3) - 10] = 10 - h_Z(3)
\]
so $h_Z(3) \leq 9$. Thus the $h$-vector of $Z$ begins $(1,3,a_2,a_3,\dots)$ where $a_2 + a_3 \leq 5$. This forces it to be one of
{\small 
\[
 (1,3,4,1, \dots, 1),\  (1,3,3,2,a_4,a_5,\dots) ,\  (1,3,2,2,a_4,a_5,\dots),  \ (1,3,2,1, \dots,1),  \ (1,3,1,\dots,1)
\] }
where the sum of all entries is $16$. These all force $Z$ to fail to impose independent conditions on quartics, giving a contradiction. Thus the general projection $\pi_P(Z)$ does not lie on a cubic curve. Then using (\ref{eq1}) we get that $\pi_P(Z)$ has $h$-vector beginning with:
$$ (1,2,3,4,2, a_5,a_6, \dots).$$
This means the possible $h$-vectors are
\[
(1,2,3,4,2,2,2), (1,2,3,4,2,2,1,1), (1,2,3,4,2,1,1,1,1).
\]
The second and third force the general projection $\pi_P(Z)$ to contain $8$ or $9$ points on a line, respectively, by Proposition \ref{elem}, which means that also $Z$ contains that number of points on a line, contradicting the assumption that $Z$ imposes independent conditions on quartics. By \cite{BGM94} Theorem 3.6, the first $h$-vector forces $13$ points on a conic (possibly reducible). By semicontinuity, there is a subset of $13$ of the points of $Z$ with property $\CC(2)$, so by Proposition \ref{CC2}, $Z$ contains $13$ points on a pair of two skew lines, again contradicting independent conditions.

\medskip

\noindent \underline{Case 2}. If $Z$ does not impose independent conditions on quartics, then $h_Z(4) < 16$ so $\dim [I_Z \cap I_P^4]_4 \geq 4$. This means 
$\dim [I_Z \cap I_P^4]_4 = \dim [I_{\pi_P(Z)}]_4 = 4$ (it cannot be $5$ since the $h$-vector cannot be zero in degree $4$).
The shortest possible $h$-vector of $\pi_P(Z)$ is then
\[
(1,2,3,4,1,1,1,1,1,1).
\]
(We do not claim that the first four entries must be as written, but anything else would be even worse from our point of view.) Thus $\pi_P(Z)$, and hence also $Z$, contains at least $10$ points on a line, so 
it imposes at most $(4+1) + 6 =11$ conditions on quartics. Thus from above, 
\[
\dim [I_{\pi_P(Z)}]_4 = \dim [I_Z \cap I_P^4]_4 = 19 - h_Z(4) \geq 19-11 = 8,
\]
which is impossible.
\end{proof}

\begin{example}
If $Z$ is a $(4,4)$-grid, then a general projection of $Z$ is a complete intersection of type $(4,4)$, and $Z$ imposes independent conditions on quartics.
Thus, as explained at the beginning of the proof of the previous proposition, $\delta \CC(4)=3$ for $Z$. Then, $(4,4)$-grids have maximal defect
$\delta \CC(4)$ among sets satisfying $\CC(4)$ and not $\CC(3)$.

On the other hand, notice that there do exist sets $Z'$ of $16$ points with $\delta \CC(4)=3$ and not satisfying $\CC(3)$, whose
general projection is not a complete intersection: for one,
let $Z'$ be a set of points in $\mathbb P^3$ consisting of $5$ general points on each of three skew lines, plus a general point. The $h$-vector of $Z'$ is $(1,3,6,3,3)$, so $\dim [I_{Z'}]_4 = 35 - 16 = 19$ ($Z'$ imposes independent conditions on quartics). The projection $\pi_P$ of $Z'$ from a general point $P$ consists of $5$ general points on each of three lines in the plane plus one general point, so $\dim [I_{Z'} \cap I_P^4]_4 = \dim [I_{\pi_P(Z')}]_4 = 2$ (an easy calculation). Then
\[
\delta \CC (4) = \dim [I_{Z'} \cap I_P^4]_4 - \dim [I_{Z'}]_4 + \binom{4+2}{3} = \dim [I_{Z'} \cap I_P^4]_4 +1 = 3.
\]
\end{example}

\begin{remark} \label{ptci}
Summarizing, there are strong links between properties $\CC(d)$ and the property that a general projection is complete
intersection, but they are conditioned on the addition of other \emph{geometric} properties, e.g. the fact that
a general projection of $Z$ is in a mild form of uniform position.

This opens an unexplored field of investigation. We do not go further in the matter, and leave it to future studies. 
\end{remark}

We also want to explore what can happen with a collection of points on a curve with respect to the property $\CC(d)$. In \cite{HMNT} the authors formed some corollaries of the result quoted in Proposition \ref{cone in P3} and Remark \ref{pts-curves} of this paper. We consolidate them here and phrase it in the language of $\CC(d)$.

\begin{corollary}
Let $C$ be a smooth, irreducible, non-degenerate curve of degree $d \geq 3$ in $\mathbb P^3$. Let $P \in \mathbb P^3$ be a general point. Let $Z \subset C$ be any set of at least $d^2+1$ points on $C$ (general or not). 

\begin{itemize}

\item[(a)] Then the cone $S_P$ over $C$ with vertex $P$ is unexpected, hence $Z$ satisfies $\CC(d)$. In fact, $S_P$  is the unique unexpected surface of this degree and multiplicity.

\item[(b)] Let $k \geq d$ be a positive integer. Then $C$  satisfies $\CC(k)$,  by taking the product of $S_P$ and an element of $[I_P^{k-d}]_{k-d}$.

\end{itemize}
\end{corollary}

We make two comments. The first is that one might hope to find an example of a set of points that project to a complete intersection of type $(r,s)$ using this last fact,  choosing $d=r$ and $k=s $ suitably. The trouble is twofold: (i) in that setting $S_P$ is always a factor of the surfaces proving $\CC(r)$ and $\CC(s)$, so we do not get that the projection is a complete intersection; (ii) if we are focused on a finite set of points $Z$ on $C$, eventually there are surfaces that contain $Z$ but not $C$, which changes the situation.

The second comment is that the bound $d^2+1$ in (a) is much bigger than necessary. We illustrate this with an example.

\begin{example}
Let $C$ be a general (ACM) smooth curve of degree $d=6$ and genus $3$ in $\mathbb P^3$. Its Hilbert function is given by

\medskip

\begin{center}

\begin{tabular}{c|cccccccccccccccccccc}
degree $t$ & 0 & 1 & 2 & 3 & 4 & 5 & 6 & 7 & 8 & \dots \\ \hline
$h_C(t)$ & 1 & 4 & 10 & 16 & 22 & 28 & 34 & 40 & 46 & \dots 
\end{tabular}

\end{center}

\medskip

\noindent The Hilbert function of a general set $Z$ of $N$ general points on $C$ has Hilbert function equal to $h_Z(t) = \min \{ h_C(t), N \}$ for each $t \geq 0$.

The above corollary says that a set $Z$ of $d^2+1 = 37$ general points on $C$ satisfies $\CC(6)$ because the cone over $C$ is an unexpected sextic for $Z$ (hence also a general projection of $Z$ lies on the sextic curve arising as the projection of $C$). In fact it is clear from the Hilbert function that for a set $Z$ of at least $34$ general points on $C$ we must have $[I_C]_6 = [I_Z]_6$, so the conclusion still holds. In particular, for $36$ general points we have only one unexpected cone, so the projection property cannot hold -- the projection is not the complete intersection of two sextics. (This is true even if the points are not general.)

But one can also move away from the condition $[I_C]_6 = [I_Z]_6$ and instead go back to the definition of unexpected surfaces. From the Hilbert function we see that for $N \leq 34$, a general set $Z$ of $N$ points on $C$ imposes independent conditions on sextics. The number of conditions expected to be imposed by a point of multiplicity $6$ is $\binom{8}{3} = 56$. Now assume $N \geq 28$. This gives
\[
\dim [I_Z]_6 - \binom{8}{3} = \binom{9}{3} - N - \binom{8}{3} \leq 0,
\] 
so the existence of $S_P$ means $Z$ satisfies $\CC(6)$, even though $[I_C]_6 \subsetneq [I_Z]_6$.

As a final observation in this example, notice that $30$ general points on $C$ have no hope of satisfying $\CC(5)$ since if it did, we could then produce infinitely many surfaces of degree $6$ and multiplicity $6$ at $P$ (by multiplying such a quintic by a general linear form vanishing at $P$), contradicting the uniqueness of $S_P$.
\end{example}

The considerations made in this paper lead us to the following conjecture, as mentioned in the introduction.

\begin{conjecture}
If $Z \subset \mathbb P^3$ is a set of $\geq 5$ points in linear general position then the general projection of $Z$ to a plane is not a complete intersection.
\end{conjecture}

\noindent On the other hand, in the appendix we show that there do exist non-grids whose general projection is a complete intersection. 
These seem to be very hard to come by!

\vfill\eject


\section{APPENDIX} \label{non-grid ex section}

\begin{center}
{\bf Unexpected examples of sets of points whose general projections are complete intersections}
\end{center}
\smallskip

\centerline{ \it by  A.\ Bernardi, L.\ Chiantini, G.\ Denham, G.\ Favacchio, B.\ Harbourne,}
\centerline{\it J.\ Migliore, T.\ Szemberg, J.\ Szpond.}
\bigskip

Unexpected cone properties for a finite set of points  are related to the property that the general projection of the points is a complete intersection, and we have seen that grids  have both properties (with a small caveat when $a=2$ -- see Theorem \ref{gridex}). In this section we focus on the latter property. This example grew out of a research work group at the Workshop on Lefschetz Properties and Jordan Type in Algebra, Geometry and Combinatorics, held  in Levico Terme, Italy, on June 25--29, 2018. 

Based on the evidence from the low degree cases given in the preceding sections, one might be tempted to conjecture that any non-degenerate set of points in $\mathbb P^3$ whose general projection to a plane is a complete intersection must, in fact, be a grid. The  purpose of this appendix is to show that it is not true, although examples seem to be hard to come by. 

Specifically, we show how to construct a set of $24$ points,  and also  subsets of $20$, $16$ and $12$ points, none of which is a grid, but all four of which have the projection property.

The paper \cite{DIV} gave the first example of an unexpected curve in the plane (although they did not use that language), and pointed out that it came from the $B_3$ line arrangement. Then \cite{HMNT} explored the surprisingly strong connection between root systems in general and unexpected hypersurfaces. In particular, they studied the root system $F_4$, which gives the following set, $Z_{F_4}$, of $24$ points in $\mathbb P^3$:

{\small
\[
\begin{array}{l}
\{ 
[1,0,0,0],[0,1,0,0],  [0,0,1,0],[0,0,0,1],[1,1,0,0], [1,-1,0,0],[1,0,1,0],  [1,0,-1,0], [1,0,0,1], 
\end{array}
\]
\[
\begin{array}{lll|}
[1,0,0,-1], [0,1,1,0],[0,1,-1,0], 
[0,1,0,1], [0,1,0,-1], [0,0,1,1], [0,0,1,-1], [1,1,1,1], 
\end{array}
\]
\[
\begin{array}{lll}
[1,1,1,-1],[1,1,-1,1],[1,-1,1,1],[1,1,-1,-1],[1,-1,1,-1],[1,-1,-1,1],[1,-1,-1,-1] 
\}.
\end{array}
\]

\noindent In \cite{HMNT} the authors checked that  $Z_{F_4}$ satisfies several properties $\CC(d)$ (without using this notation). Indeed, they experimentally computed the following list of dimensions of the systems of cones through $Z_{F_4}$ with vertex at a general point $Q\in\pp^3$:

\medskip

\begin{center}\begin{tabular}{c|cccccccc}
$d$          & 3 & 4 & 5 & 6 & 7 & 8&  \dots \\ \hline
dim          & 0 & 1 & 3 & 7 & 13 & 21 &  \dots \\ \hline
expected & 0 & 0 & 0 & 4 & 12 & 21 &  \dots \\ \hline
                 &  & unexp.  &unexp. & unexp. & unexp. &  &  \dots 
\end{tabular} 
\end{center}

\medskip

\noindent It follows that $Z_{F_4}$ satisfies $\CC(4),\CC(5),\CC(6),\CC(7)$. 
One can check by hand (or computer) that $Z_{F_4}$ is not a $(4,6)$-grid. One can then verify on the computer that the general projection of $Z_{F_4}$ is a complete intersection of type $(4,6)$.

It would be good to have a theoretical argument for the fact that the projection is a complete intersection. Here is a start.

   Let $Q=(a(0):a(1):a(2):a(3))$ be a general point in $\mathbb P^3$.
   Let $I=\left\{j,k,l\right\} \subset\left\{0,1,2,3\right\}$,  with $i,j,k$ distinct and let
   $$f_i=3a(j)(a(k)^2-a(l)^2)x(j)^2x(k)x(l)
        +3a(k)(a(l)^2-a(j)^2)x(j)x(k)^2x(l)
   $$$$     +3a(l)(a(j)^2-a(k)^2)x(j)x(k)x(l)^2$$
   $$   +a(j)^3(x(k)^3x(l)-x(k)x(l)^3)
        +a(k)^3(x(j)x(l)^3-x(j)^3x(l))
        +a(l)^3(x(j)^3x(k)-x(j)x(k)^3),$$
   where $i\in\left\{0,1,2,3\right\}\setminus I$.
   
   This is the unique curve in $\mathbb P^2(x(j):x(k):x(l))$ vanishing at the $B_3$ configuration
   of points and having a point of multiplicity $3$ at $(a(j):a(k):a(l))$, see \cite[Section 2]{BMSS}.
   It can be also viewed as a cone in $\mathbb P^3(x(0):x(1):x(2):x(3))$.
   
   It can be easily checked that the quartic
   $$f=-a(0)f_0+a(1)f_1-a(2)f_2+a(3)f_3$$
   vanishes to order $4$ at $Q$ and vanishes at all points of $Z_{F_4}$. Since the expected dimension of the quartic cones containing $Z_{F_4}$ with general vertex is $0$, this shows that the points of  $Z_{F_4}$ satisfy $\CC(4)$. 
   
}

In more detail, it turns out (experimentally, using CoCoA \cite{cocoa}) that  $Z_{F_4}$ has $18$ sets of $4$ collinear points, and no subset of $\geq 5$ collinear points. Hence it is not a $(4,6)$-grid.

Now, in general suppose $X$ is a complete intersection of type $(4,d)$ in the plane, defined by forms $G_4, G_d$ respectively. Assume that $X$ contains a subset of four collinear points, $X_1$, defined by forms $L, F$ of degrees $1$, $4$ respectively. By Proposition \ref{linkci}, as long as $G_4$ is not a multiple of $L$, the residual $X \backslash X_1$ is a complete intersection of type $(4,d-1)$. 

Motivated by this fact, experimentally we were able to remove $4$ collinear points from $Z_{F_4}$ to produce a subset $Z_1$ of $20$ points that is not a grid, and remove a second set of $4$ collinear points to produce a subset $Z_2$ of $16$ points that is not a grid, but with the property that the general projection of $Z_1$ is a complete intersection of type $(4,5)$ and the general projection of $Z_2$ is a complete intersection of type $(4,4)$. For instance, we can take
{\small 
\[
Z_1 = \{ [1, 0, 0, 0], [0, 1, 0, 0], [1, 1, 0, 0], [1, -1, 0, 0], [1, 0, 1, 0], [1, 0, -1, 0], [1, 0, 0, 1], [1, 0, 0, -1],
\]
\[
 [0, 1, 1, 0], [0, 1, -1, 0], [0, 1, 0, 1], [0, 1, 0, -1], [1, 1, 1, 1], [1, 1, 1, -1], [1, 1, -1, 1], [1, -1, 1, 1], 
 \]
 \[[1, 1, -1, -1], [1, -1, 1, -1], [1, -1, -1, 1], [1, -1, -1, -1] \}
\]
}
and
{\small 
\[
Z_2 = \{ [0,1,0,-1],[0,1,0,0],[0,1,0,1], [0,1,1,0],[1,-1,0,0],[1,0,-1,0],[1,0,0,0], [1,0,0,1], 
\]
\[
[1,0,1,0], [1,1,0,0],[1,-1,-1,-1], [1,-1,-1,1],[1,-1,1,1],[1,1,-1,1],[1,1,1,-1],[1,1,1,1] \}.
\]
}
($Z_2$ has only four sets of four collinear points, hence is not a $(4,4)$-grid. Note that this is not the $B_4$ configuration, and in fact the $B_4$ configuration does not have the projection property.) 

Unfortunately the same trick does not work one more time: every removal of four collinear points from $Z_2$ results in a $(3,4)$-grid.
However, in \cite{HMNT} the authors also considered the root system $D_4$, which corresponds to the set of $12$ points 
{\small
\[
Z_3 =  \{ 
[1,1,0,0], [1,-1,0,0],[1,0,1,0],  [1,0,-1,0], [1,0,0,1], 
[1,0,0,-1], [0,1,1,0],[0,1,-1,0], 
\]
\[
[0,1,0,1], [0,1,0,-1], [0,0,1,1], [0,0,1,-1]
\}.
\]
}
It was noted in \cite{HMNT} that $Z_3$ satisfies $\CC(3)$ and $\CC(4)$, and we checked on the computer that the general projection is a complete intersection and that $Z_3$ is not a $(3,4)$-grid. Notice that $Z_3$ is contained in $Z_{F_4}$, even if it is not contained in $Z_2$. These $12$ points have $16$ subsets of $3$ collinear points, but no subset of $4$ collinear points. It is also possible to find other sets of $12$ points that project to a complete intersection of type $(3,4)$. For example,
the points in the following list
\small{
\[
Z_4 = \{ 
[1,1,1,1]
[1,1,1,-1],
[1,1,-1,1],
[1,1,-1,-1],
[1,-1,1,1],
[1,-1,1,-1], 
\]
\[
\ \ \ \ \ \ [1,-1,-1,1],
[1,-1,-1,-1],
[1,0,0,0],
[0,1,0,0],
[0,0,1,0],
[0,0,0,1]
\}
\]
}have this property. However, one can check that it is projectively equivalent to the $D_4$ configuration via the matrix
\[
\left [
\begin{array}{cccc}
1& 0 & 0 & 1 \\
1 & 0 & 0 & -1 \\
0 & 1 & 1 & 0 \\
0 & 1 & -1 & 0
\end{array}
\right ]
\]
(which sends $Z_4$ to $Z_3$ and also sends $Z_3$ to $Z_4$).

\begin{question} \label{LGP question}
Are there any examples of $ab$ points in $\mathbb P^3$, other than $ab = 12$, $16$, $20$ or $24$, that are not $(a,b)$-grids but that have a general projection that is a complete intersection of type $(a,b)$?
\end{question}

\vfill\eject

\end{document}